\documentclass[letterpaper, 10 pt, conference]{ieeeconf}  

\IEEEoverridecommandlockouts                              

\usepackage{amsmath, amssymb, amscd, amsthm, amsfonts}
\usepackage{graphicx}
\usepackage{hyperref}
\usepackage{tikz}
\usepackage{blkarray}
\usepackage{accents}
\usepackage{algorithmic}
\usepackage{booktabs}
\usepackage[linesnumbered,ruled,vlined]{algorithm2e}
\usepackage[normalem]{ulem}
\usepackage{dsfont}
\usepackage{caption}
\usepackage{subcaption}

\usepackage{amsmath}
\DeclareMathOperator*{\argmin}{argmin}

\usetikzlibrary{arrows.meta,positioning}

\def\sS{\mathsf{S}}

\def\sI{\mathsf{I}}
\def\sR{\mathsf{R}}

\def\blambda{\boldsymbol{\lambda}}

\newcommand\bp{\boldsymbol{p}}
\newcommand\bu{\boldsymbol{u}}
\newcommand\btheta{\boldsymbol{\theta}}
\newcommand\bpsi{\boldsymbol{\psi}}
\newcommand\bvarphi{\boldsymbol{\varphi}}

\newtheorem{theorem}{Theorem}[section]

\newtheorem{definition}[theorem]{Definition}
\newtheorem{lemma}[theorem]{Lemma}
\newtheorem{proposition}[theorem]{Proposition}

\let\oldsout\sout
\renewcommand{\sout}[1]{\textcolor{red}{\oldsout{#1}}}

\title{\LARGE \bf
Epidemic Control using Stackelberg Mean Field Game \\ with Budget Constraint
}

\author{Huaning Liu and G\"ok{\c c}e Dayan{\i}kl{\i}
\thanks{This work is supported by National Science Foundation under grant DMS-2436332.}
\thanks{Both authors are with the Department of Statistics,
  University of Illinois Urbana-Champaign, 
  Champaign, IL 61820, USA. {\tt\small huaning3@illinois.edu}, 
        {\tt\small gokced@illinois.edu}}
}

\begin{document}

\maketitle
\thispagestyle{empty}
\pagestyle{empty}

\begin{abstract}
Epidemic mitigation comes with a price in practice, and when budget is limited, its effective allocation becomes an important question on public policy. In this paper, we study epidemic control through a Stackelberg mean-field game in which a principal, representing the government, chooses social distancing guidelines and vaccination levels for a large population of rational minor agents subject to a budget-spending process. The principal anticipates the minor agents mean-field Nash equilibrium (MFNE) and selects intervention policies to minimize her own objective, inducing a bi-level optimal control problem. We define the Stackelberg mean-field game equilibrium (SMFE) and give its necessary condition by a forward-backward ordinary differential equation (FBODE) system, and provide existence and uniqueness results. We further propose an iterative fixed-point algorithm to solve the FBODE numerically and present computational observations that illustrate the resulting equilibrium behavior.
\end{abstract}

\section{Introduction}
Since the COVID-19 pandemic, regulators have experimented with different mitigation policies, from strict lockdowns to light regulations. These choices produced a wide range of public-health and economic outcomes, revealing the trade-off that stronger interventions may suppress transmission more effectively, but often at substantial social and economic costs~\cite{horrible2022}. Therefore, one of the main aims of epidemic control is to find the optimal incentives for people to decrease the spread of the epidemic while balancing out these different outcomes~\cite{elie2020contact,aurell2022optimal}.

Control literature that studied this question has been built based on the classical susceptible-infected-recovered (SIR) model and its variants, where a centralized planner selects intervention intensities to minimize an aggregate social cost. Recent extensions incorporated additional constraints and objectives, such as limited ICU availability via a constraint on infection levels~\cite{Rossa_2024}, minimizing the time required to reduce infections below a target threshold~\cite{Bolzoni_2017}, or budget constraints in response to an immediate epidemic outbreak~\cite{Salcedo2020}. A common feature of these models is that the population is assumed to comply perfectly with the imposed policy.

A line of work complementary to this feature studies human response in epidemics modeling through game theory, where individuals choose actions such as socialization or vaccination to optimize their own objectives for exogenously given policy of the regulator. In large populations, these strategic interactions are naturally modeled by mean-field games (MFGs). Following the seminal works in~\cite{Lasry_2007, Huang_2006}, the MFG framework enables mean field approximations as the population size approaches infinity and reduces the Nash equilibrium characterization to the analysis of a coupled forward-backward differential equations system by assuming identical and infinitesimal agents. We refer the readers to~\cite{cho2020mean,elie2020contact,aurell2022finite,Doncel_Gast_Gaujal_2022,olmez2022modeling,pnas_epidemics,liu2025incorporatingauthorityperceptioneconomic} for more works on epidemic modeling under large populations. However, none of these works have the optimal policy-making by modeling the regulator as a decision-maker optimizing some objectives.

To model the hierarchical structure where a regulator would like to choose an optimal policy while anticipating individual responses from the population, we need to move beyond standard MFG and consider Stackelberg MFG formulation, whose theoretical foundations are initially studied in~\cite{Elie2018ATail,carmonawang2021finite}. In this paper, we study epidemic mitigation as a Stackelberg MFG with an explicit state dynamics for the regulator which represents their budget process. The principal chooses time-dependent social distancing guidelines and vaccination controls to minimize a budget-aware objective, where minor agents respond with the MFG equilibrium where the dynamics are inspired from SIR model.

Our contributions are three-fold: we formulate the first continuous-time Stackelberg MFG \textit{with principal state process} with a motivation to model epidemic mitigation problem; we give necessary conditions for the Stackelberg mean-field equilibrium by a forward-backward ordinary differential equation (FBODE) system and show existence and uniqueness results for the FBODE system. Finally, we develop a numerical algorithm to solve the system and study the equilibrium behavior.

We highlight~\cite{aurell2022optimal} as a related prior study on the application of Stackelberg MFG to epidemic control. Our work differs from that paper in two main respects. First, we incorporate the budget dynamics as an explicit state process for the principal and add vaccination as her another decision variable, which allows us to study the allocation of resources between social distancing and vaccination. Second,~\cite{aurell2022optimal} relies on Sannikov’s approach~\cite{Sannikov2008AContinuous} for the characterization of Stackelberg equilibrium; whereas our method is inspired by the approach of~\cite{Bensoussan_2015} and gives the characterization from an optimal control perspective, treating the minor agents' value function and density flow as part of the principal’s state variables.

The rest of the paper is organized as follows. In Section~\ref{sec:sec_model_formulation}, we introduce a general epidemic Stackelberg MFG with principal state dynamics and review existing results on the Nash equilibrium of the minor agent MFG to ensure the well-posedness of the Stackelberg problem. The main theoretical results on the Stackelberg MFG equilibrium are presented in Section~\ref{sec:smfe_charac}. Numerical algorithms and experiments are discussed in Section~\ref{sec:simulation_study}. Finally, Section~\ref{sec:future_dir} concludes with a discussion of future directions.

\section{Model Formulation}
\label{sec:sec_model_formulation}
In our model, the individuals in the population are the minor agents and the government is in the regulator (i.e., principal) role. Individuals will respond to the principal with a Nash equilibrium that will be approximated via MFG formulation and the principal will optimize the policies by taking into account the Nash equilibrium response of the individuals. We first start with stating the MFG model of the minor agents and then introduce the Stackelberg MFG model with the inclusion of the principal. 
\subsection{Minor Agent's Game}
We first introduce and analyze epidemic MFG over a finite time horizon $T>0$ given the incentives or policies of the government (i.e., principal). As common in MFG models, we can directly focus on the model of the representative agent and her interactions with the population via the mean field. At time $t \in[0, T]$, a representative agent is in a health state $X_t \in E:=\{\sS, \sI, \sR\}$ which represents \textit{susceptible}, \textit{infected}, and \textit{recovered}. She chooses a state-dependent socialization level $\alpha_t(e) \in[0,1]$ for each $e \in E$. 
The controlled health state process $(X_t)_{t \in [0,T]}$ follows a continuous-time Markov chain, and an agent’s transition rate from state $\sS$ to $\sI$ depends both on her control and on the mean field interactions through the average socialization level of infected population $Z_t$. Namely, if we denote the population joint control-state distribution as $\rho_t$, then we have $Z_t = \int_A a \rho_t(d a, \sI)$. Vaccination transitions an individual from state $\sS$ to state $\sR$ at a rate $v_t$ and will be controlled by the principal and when analyzing the MFG for minor agents, it can be assumed to be given exogenously. We denote the base transmission rate as $\beta > 0$ and further assume that the transition from state $\sI$ to $\sR$ (recovery) and the transition from state $\sR$ to $\sS$ (waning of immunity) happens after exponentially distributed times, with constant rates $\gamma > 0$ and $\eta \geq 0$, respectively. The state process of representative agent is summarized in Figure~\ref{fig:state_diagram} and its transition-rate matrix thereby writes
$$Q(t,\alpha,Z)=\left(\begin{array}{ccc} \cdots & \beta \alpha_t(\sS) Z_t & v_t \\ 0 & \cdots & \gamma \\ \eta & 0 & \cdots\end{array}\right).$$

The notation $\cdots$ represents the negative of the sum of
the elements on the same row to satisfy the condition of
having the row sum equal to 0. 
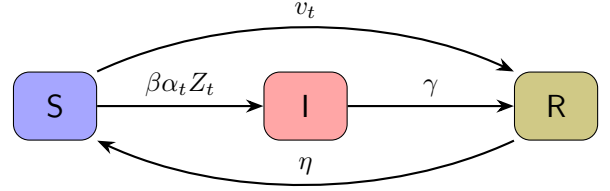
\begin{figure}[h]
  \centering
\begin{tikzpicture}[
  >=Stealth,
  node distance=2.2cm,
  state/.style={
    draw,
    rounded corners=6pt,
    minimum width=11mm,
    minimum height=9mm,
    inner sep=2pt,
    font=\large\bfseries
  }
]

\node[state, fill=blue!35] (S) {$\sS$};
\node[state, fill=red!35, right=of S] (I) {$\sI$};
\node[state, fill=olive!45, right=of I] (R) {$\sR$};

\draw[->, thick]
  (S.north east) to[out=25, in=155, looseness=0.85]
  node[midway, above] {$v_t$}
  (R.north west);

\draw[->, thick] (S) -- (I)
  node[midway, above] {$\beta\alpha_tZ_t$};

\draw[->, thick] (I) -- (R)
  node[midway, above] {$\gamma$};

\draw[->, thick]
  (R.south west) to[out=-155, in=-25, looseness=0.85]
  node[midway, above] {$\eta$}
  (S.south east);
\end{tikzpicture}
\caption{State flow diagram of the representative minor agent}
\label{fig:state_diagram}
\end{figure}
The representative agent would like to minimize the following cost functional by choosing her socialization level $\alpha = (\alpha_t)_{t \in [0,T]}$, given the mean field interactions $Z = (Z_t)_{t \in [0,T]}$:
{\small
\begin{equation}
\label{eq:minor_cost}
\begin{aligned}
J\big(\alpha;&Z\big)
= \mathbb{E}\Big[
\int_{0}^{T} \Big[
\big(\lambda_t^{\sS}-\alpha_t\big)^2\mathds{1}_{\sS}\big(X_t\big) \\
&+ \Big(\big(\lambda_t^{\sI}-\alpha_t\big)^2 + c_{\sI}\Big)\mathds{1}_{\sI}\big(X_t\big)
+ \big(\lambda_t^{\sR}-\alpha_t\big)^2 \mathds{1}_{\sR}\big(X_t\big)
\Big]\,dt
\Big].
\end{aligned}
\end{equation}
}
Here, $\mathds{1}_e\left(X_t\right)$ represents the indicator function for $e \in E$ that is equal to 1 if $X_t=e$ and 0 otherwise. A time-dependent social distancing guideline $\left(\lambda_t^e\right)_{e \in E, t \in[0, T]}$ is announced by the principal, and agents incur opportunity or penalty costs when they deviate from it. The parameter $c_{\sI}>0$ denotes the infection cost. In addition, we assume that the admissible controls of minor agents are closed-loop; that is, Markovian, square-integrable, and $[0,1]$-valued functions defined on $[0,T]\times E$. We thereby formulate an extended MFG for minor agents.

\begin{definition}
    The profile $(\hat{\alpha}, \hat{Z}) = (\hat{\alpha}_t, \hat{Z}_t)_{t \in [0,T]}$ is an MFG Nash equilibrium (MFNE) if for any admissible $\alpha$, we have
    $J(\hat{\alpha};\hat{Z}) \leq J(\alpha;\hat{Z}),$
    where $\hat{Z}_t = \int_A a \hat{\rho}_t(d a, \sI)$ and $\hat{\rho} = (\hat{\rho}_t)_{t \in [0,T]}$ is the action-state distribution law induced by $\hat{\alpha}$.
\end{definition}

Finite state MFGs have been extensively studied, and equilibrium characterization as well as existence and uniqueness results are established (see e.g.,~\cite{Gomes2013,carmonawang_prob_approach_extended}) for non-extended settings. In the epidemiological setting which requires extended finite-state MFGs, related formulations have been analyzed in~\cite{liu2025incorporatingauthorityperceptioneconomic} for a multi-population model and in~\cite{aurell2022finite} within a graphon game framework. To ensure that the Stackelberg formulation introduced below is well-posed, we briefly restate the key results needed. Denote the population state density flow as $(p_t(e))_{e \in E, t \in [0,T]}$, where $p_t \in \Delta_p := \{(p^\sS, p^\sI, p^\sR) \in[0,1]^3: p^\sS+p^\sI+p^\sR=1\}$. Let $\pi_0 \in \Delta_p$ be the \textit{known} initial state distribution. The value function of the representative agent is defined as
{\small $$\begin{aligned}
u_t(e)
&:=\inf_{(\alpha_\tau)_{\tau\in[t,T]}}
\mathbb{E}\Big[
\int_t^T \Big[
\big(\lambda_\tau^{\sS}-\alpha_\tau\big)^2\mathbf{1}_{\sS}(X_\tau)\\
&\hspace{-0.8cm}+\big((\lambda_\tau^{\sI}-\alpha_\tau)^2+c_{\sI}\big)\mathbf{1}_{\sI}(X_\tau)
+\big(\lambda_\tau^{\sR}-\alpha_\tau\big)^2\mathbf{1}_{\sR}(X_\tau)
\Big]d\tau
\Bigm| X_t=e
\Big].
\end{aligned}$$}
We first present a forward-backward ordinary differential equations (FBODE) system that characterizes the MFNE.
\begin{proposition}
\label{prop:minor_charac}
The MFNE control profile is described
{\small
\begin{equation*}
\label{eq:minor_control}
\begin{aligned}
\hat{\alpha}_t(\sS)&=\lambda_t^{\sS}+\frac{\beta Z_t\big(u_t(\sS)-u_t(\sI)\big)}{2},\\
\hat{\alpha}_t(\sI)&=\lambda_t^{\sI},
\qquad\hat{\alpha}_t(\sR)=\lambda_t^{\sR},
\end{aligned}
\end{equation*}}
where $(u,p)$ satisfy the coupled system:
{\small
\begin{align}
\dot p_t(\sS)&=-\beta\hat{\alpha}_t(\sS)Z_tp_t(\sS)-v_tp_t(\sS)+\eta p_t(\sR),\nonumber\\
\dot p_t(\sI)&=\beta\hat{\alpha}_t(\sS)Z_tp_t(\sS)-\gamma p_t(\sI),\nonumber\\
\dot p_t(\sR)&=\gamma p_t(\sI)+v_tp_t(\sS)-\eta p_t(\sR),\label{eq:minor_KFP}
\end{align}}
and
{\small
\begin{align}
\dot u_t(\sS)&=\beta\hat{\alpha}_t(\sS)Z_t\big(u_t(\sS)-u_t(\sI)\big)+v_t\big(u_t(\sS)-u_t(\sR)\big)\nonumber\\
&\hskip5mm-\big(\lambda_t^{\sS}-\hat{\alpha}_t(\sS)\big)^2\nonumber\\
\dot u_t(\sI)&=\gamma\big(u_t(\sI)-u_t(\sR)\big)-c_{\sI},\nonumber\\
\dot u_t(\sR)&=\eta\big(u_t(\sR)-u_t(\sS)\big),\label{eq:minor_HJB}
\end{align}}
with terminal and initial conditions
{\small
\begin{equation*}
u_T(e)=0,\qquad p_0(e)=\pi_0^e,\qquad \forall e\in\{\sS,\sI,\sR\},
\end{equation*}}
and the consistency condition $Z_t=\hat{\alpha}_t(\sI)p_t(\sI), \forall t\in[0,T]$.
\end{proposition}
\begin{proof}
    The proof relies on Hamilton-Jacobi-Bellman (HJB) equation induced from dynamic programming principle and Kolmogorov-Fokker-Planck (KFP) equation. Interested readers may refer to~\cite{liu2025incorporatingauthorityperceptioneconomic} for details in a generalized setup.
\end{proof}

\begin{proposition}
    \label{prop:minor_well_posedness}
    Under short time horizon condition, the FBODE system in~\ref{prop:minor_charac} admits a unique bounded solution.
\end{proposition}
\begin{proof}
    The existence and uniqueness can be established based on Banach fixed point theorem following similarly to the result introduced in~\cite{liu2025incorporatingauthorityperceptioneconomic}.
\end{proof}
We remark that, in much of the Stackelberg game literature addressing more general settings (e.g.,~\cite{Elie2018ATail,Bensoussan_2015,aurell2022optimal}), uniqueness of the followers' equilibrium is stated as an assumption directly to define the admissible controls of the principal. In our setting, we explicitly establish this property to ensure that the principal’s problem is indeed well-posed. Next, we formulate the principal's problem.

\subsection{Stackelberg Mean-Field Game with State Processes}
As described above, minor agents take the time-dependent social distancing guideline $(\lambda_t^e)_{t\in[0,T],e\in E}$ and the vaccination rate $(v_t)_{t\in[0,T]}$ exogenously while finding their Nash equilibrium. We now introduce in the model a principal who selects these two policy profiles to optimize her own objective subject to some budget dynamics. This yields a bi-level Stackelberg MFG with additional state dynamics for the principal.

We consider open-loop control for the principal: The information available to her is restricted to the initial condition and model primitives, without feedback from the population state evolution, and her controls are therefore adapted to the filtration generated at time $0$. We denote the principal's policy as $\theta := (\theta_t)_{t \in [0,T]}$ with $\theta_t := ((\lambda_t^e)_{e \in E},v_t)$.
\begin{definition}
    A principal policy $\theta$ is \textit{admissible} if $T$ satisfies the short time horizon condition in Proposition~\ref{prop:minor_well_posedness}, i.e. each policy uniquely induces a MFNE.
\end{definition}
The condition in Proposition~\ref{prop:minor_well_posedness} requires uniform upper bounds on $\lambda$ and $v$. Accordingly, we restrict $\lambda_t$ to take values in $[0,1]$, and $v_t$ to lie in $[0,V]$ for some $V > 0$. 
Under these bounds, the principal's admissible policy space is well defined for a fixed short enough time horizon, and we define it as
$$\Theta:=\big\{\theta :[0, T] \rightarrow [0,1]^3 \times [0, V]: \theta \text { measurable} \big\}.$$
The principal is also subject to a budget state process $B := (B_t)_{t \in [0,T]}$, whose dynamics is given by
\begin{equation}
    \label{eq:budget_sde}
    dB_t = b(t,\theta_t,p_t^{\theta})dt, \quad B_0 = b_0.
\end{equation}
Here $p_t^{\theta}$ is the (unique) MFNE state density flow induced by the policy $\theta = (\lambda,v)$. The drift $b:[0,T] \times[0,1]^3 \times [0, V] \times \Delta_p \rightarrow \mathbb{R}$ represents the budget spending rate. The term $b_0 > 0$ is the \textit{known} total initial budget amount. The principal chooses her policy to minimize her own cost
{\small \begin{equation}
    \label{eq:principal_obj}
    \begin{aligned}
        J_{G}(\theta) & = \int_{0}^{T} \Big[\Phi_{p}(p_t^{\theta}(\sI)) + \Phi_{\lambda}(\lambda_t) + \Phi_{v}(v_t)\Big]dt + \Psi(B_T).
    \end{aligned}
\end{equation}}
Here, $\Phi_p:[0,1] \to [0,\infty)$, $\Phi_{\lambda}: [0,1]^3 \rightarrow [0,\infty)$, and $\Phi_{v}: [0,V] \rightarrow [0,\infty)$ are running cost functions representing penalization on high infection proportion in the population, and costs related to implemented social distancing and vaccination policies. We impose the budget constraint in a \textit{soft} form: the principal is allowed to run a budget deficit while mitigating the epidemic, but any negative terminal budget at time $T$ is penalized through a terminal cost. This penalty is captured by a function $\Psi: \mathbb{R} \rightarrow [0,\infty)$. We have now formulated the principal's problem and are at the stage of stating the equilibrium notion.
\begin{definition}
    \label{def:stackelberg_mfnash}
 We call a principal policy $\theta^* = (\lambda^*,v^*)$ a Stackelberg MFG equilibrium (SMFE) if $(\lambda^*,v^*) \in \argmin_{\Theta} J_{G}(\lambda,v)$.
\end{definition}

\section{Main Theoretical Results}
\label{sec:smfe_charac}
We begin with a remark on the difference of Nash and Stackelberg equilibria between a principal and a large population of agents. The former refers to a notion where both the principal and the minor-agent population choose strategies that are mutual best responses to each other. In contrast, a Stackelberg equilibrium is inherently hierarchical: the principal selects a policy to optimize her objective, taking into account the population (Nash) equilibrium dynamics induced by that policy. Therefore, our interest lies in a bi-level optimization formulation, which will be recast as an optimal control problem.

In order to give a characterization of the Stackelberg equilibrium, we will introduce an explicit form for the principal's model. We specify $\Phi_{p}(p_t(\sI)) := K^p p_t(\sI)^2$, $\Phi_\lambda(\lambda_t) := \sum_{e\in E}\frac12 K^e\bigl(\lambda_t^e-\bar\lambda^e\bigr)^2$, $\Phi_v(v_{t}) := \frac12 K^{v} (v_t - \bar{v})^2$ and $\Psi(B_T)
:= K^b \bigl(\min\{0,B_T\}\bigr)^2$. Here $K^p > 0$, $K^e > 0$, $K^v>0$ and $K^b>0$ are scaling factors for the corresponding cost items. Also, $\boldsymbol{\bar\lambda} = (\bar\lambda^e)_{e \in E}$ with $\bar\lambda^e > 0$ and $\bar{v} > 0$ denote the public health authority's (such as from C.D.C.) recommended levels of socialization and vaccination, respectively, issued to the government. Penalizing deviations of the government's policy from $\boldsymbol{\bar\lambda}$ and $\bar{v}$ captures the reputational cost associated with adopting policies that are perceived as either overly lenient or overly stringent. 
For vaccination, $\bar{v}$ may alternatively represent the vaccination level expected by vaccine suppliers, in which case deviations may signal policy intervention and distort long-run supply incentives. This cost functional formulation is motivated by the findings from the public health policy literature. In particular, existing studies highlight the importance of government trust~\cite{Han2021Trust}, the potential for stay-at-home orders to improve public favor towards the government~\cite{Nelson2020Gov}, and the role of external experts in fostering trust during crises~\cite{Broekema2018Crisis}.
To avoid double counting, we defer all financial interpretations to the budget dynamics.
The budget dynamics is specified as
\begin{equation}
\label{eq:budget_spec}
dB_t
=
\big[
\sum_{e\in E} - m^e(1 - \lambda_t^e) p_t(e)
- m^{v} v_t p_t(\sS)
\big]dt,
\end{equation}
where $m^e > 0$ and $m^v > 0$ are scaling parameters, $e \in E$. For social distancing, as suggested by the model formulation, the value $1$ corresponds to no regulation; hence, stricter social distancing policies weighted by the population size to regulate in the corresponding state contribute higher budget expenditure. Similarly, the principal's vaccination-related spending scales with both the vaccination level and the proportion of susceptible individuals.

We characterize the SMFE using an approach inspired by~\cite{Bensoussan_2015}. In~\cite{Bensoussan_2015}, a Stackelberg game with one principal and only \textit{one} follower is studied by treating the follower's forward state and backward adjoint processes, which characterize the follower's optimality, as the state variables of the principal. We extend this idea to the MFG setting with extended interactions and incorporate an additional state variable that represents the principal's budget. In particular, we regard the state-density flow and the value function of the representative minor player, together with the budget process, as the state variables in the principal's control problem, and analyze the principal optimality using optimal control techniques. Since the principal's admissible policy is open-loop in this model, the Pontryagin Maximum Principle applies. Moreover, thanks to the finite-state MFG formulation for minor agents, no variational derivative with respect to the state density is required. This enables us to derive a semi-explicit FBODE system that characterizes the SMFE. To save space and simplify the notations, we write the state as a superscript for $p$, $u$, and $\alpha$ in the following theorem.
\begin{theorem}
    {Suppose the control tuple
    $(\lambda_t^{\sS *},\lambda_t^{\sI *},
    \lambda_t^{\sR *},v_t^*)_{t\in[0,T]}$
    is an SMFE. Then it necessarily satisfies the following equations
    }
    \small
    \begin{equation}
        \label{eq:smfe_control}
        \begin{aligned}
K^\sS \bar{\lambda}^\sS-m^\sS p_t^\sS \chi_t &= K^\sS \textcolor{blue}{\lambda_t^{\sS *}} \\ & \hspace{-0.5cm}+\big[\beta p_t^\sI p_t^\sS(\varphi_t^\sI-\varphi_t^\sS)+\beta p_t^\sI(u_t^\sS-u_t^\sI) \psi_t^\sS\big]\textcolor{blue}{\lambda_t^{\sI *}}\\
K^\sI \bar{\lambda}^\sI-m^\sI p_t^\sI \chi_t & = \big[\beta p_t^\sI p_t^\sS(\varphi_t^\sI-\varphi_t^\sS)+\beta p_t^\sI(u_t^\sS-u_t^\sI) \psi_t^\sS\big]\textcolor{blue}{\lambda_t^{\sS *}} \\ & \hspace{-2.55cm} + \big[K^\sI + \beta^2 (p_t^\sI)^2 p_t^\sS (u_t^\sS-u_t^\sI)(\varphi_t^\sI-\varphi_t^\sS)+\frac{1}{2}\beta^2 (p_t^\sI)^2(u_t^\sS-u_t^\sI)^2 \psi_t^\sS\big]\textcolor{blue}{\lambda_t^{\sI *}} \\
\textcolor{blue}{\lambda_t^{R*}} &= \bar{\lambda}^\sR-\frac{m^\sR p_t^\sR\chi_t}{K^\sR} \\
\textcolor{blue}{v_t^*} &= \bar{v} + \frac{p_t^\sS}{K^v} (\varphi_t^\sS - \varphi_t^\sR) - \frac{u_t^\sS - u_t^\sR}{K^v}\psi_t^\sS + \frac{m^v}{K^v} p_t^\sS \chi_t
\end{aligned}
    \end{equation}
\normalsize
where the state processes $(u_t, p_t, B_t)_{t \in [0,T]}$ and the corresponding adjoint processes $\varphi = (\varphi_t^\sS,\varphi_t^\sI,\varphi_t^\sR)_{t \in [0,T]}$, $\psi = (\psi_t^\sS,\psi_t^\sI,\psi_t^\sR)_{t \in [0,T]}$ and $\chi = (\chi_t)_{t \in [0,T]}$ satisfy the following FBODE system
\small
\allowdisplaybreaks
\begin{align}
        \dot p_t^\sS&=-\beta\hat{\alpha}_t^\sS\lambda_t^{\sI*}p_t^\sI p_t^\sS-v^*_t p_t^\sS+\eta p_t^\sR,\nonumber\\
\dot p_t^\sI&=\beta\hat{\alpha}_t^\sS \lambda_t^{\sI*}p_t^\sI p_t^\sS-\gamma p_t^\sI,\nonumber\\
\dot p_t^\sR&=\gamma p_t^\sI+v_t^*p_t^\sS-\eta p_t^\sR,\nonumber \\
\dot u_t^\sS&=\beta\hat{\alpha}_t^\sS \lambda_t^{\sI*}p_t^\sI\big(u_t^\sS-u_t^\sI\big)+v^*_t\big(u_t^\sS-u_t^\sR\big) - \big(\lambda_t^{\sS*}-\hat{\alpha}_t^\sS\big)^2,\nonumber\\
\dot u_t^\sI&=\gamma\big(u_t^\sI-u_t^\sR\big)-c_{\sI},\nonumber\\
\dot u_t^\sR&=\eta\big(u_t^\sR-u_t^\sS\big),\nonumber\\
\dot B_t &= \sum_{e\in E} - m^e(1 - \lambda_t^{e*}) p_t^e
- m^{v}v^{*}_{t} p_t^\sS, \nonumber\\
\dot{\varphi}_t^\sS&=\beta\hat{\alpha}_t^\sS \lambda_t^{\sI*} p_t^\sI (\varphi_t^\sS - \varphi_t^\sI)+v_t^*(\varphi_t^\sS - \varphi_t^\sR)\nonumber\\&\hspace{3cm}+\big(m^\sS(1-\lambda_t^{\sS*})+ m^v v_t^*\big)\chi_t, \nonumber\\
\dot{\varphi}_t^\sI&=-2 K^p p_t^\sI+\Big(\beta \lambda_t^{\sS*} \lambda_t^{\sI*}+\beta^2(\lambda_t^{\sI*})^2 p_t^\sI (u_t^\sS-u_t^\sI)\Big) p_t^\sS (\varphi_t^\sS-\varphi_t^\sI)\nonumber\\&\quad+\gamma(\varphi_t^\sI-\varphi_t^\sR)
-\Big(\beta \lambda_t^{\sS*} \lambda_t^{\sI*} (u_t^\sS-u_t^\sI) \nonumber\\ &\quad +\frac{\beta^2}{2}(\lambda_t^{\sI*})^2 p_t^\sI(u_t^\sS-u_t^\sI)^2\Big)\psi_t^\sS+ m^\sI (1-\lambda_t^{\sI*})\chi_t, \nonumber\\
\dot{\varphi}_t^\sR&=-\eta(\varphi_t^\sS-\varphi_t^\sR)+ m^\sR(1-\lambda_t^{\sR*})\chi_t,\nonumber \\
\dot{\psi}_t^\sS&=\frac{\beta^2}{2}(\lambda_t^{\sI*})^2 (p_t^\sI)^2 p_t^\sS (\varphi_t^\sS-\varphi_t^\sI)\nonumber \\ &-\Big[\beta \lambda_t^{\sS*} \lambda_t^{\sI*} p_t^\sI+\frac{\beta^2}{2}(\lambda_t^{\sI*})^2 (p_t^\sI)^2 (u_t^\sS-u_t^\sI)\Big] \psi_t^\sS+\eta \psi_t^\sR - v^*_t \psi_t^\sS,\nonumber \\
\dot{\psi}_t^\sI&=-\frac{\beta^2}{2}(\lambda_t^{\sI*})^2 (p_t^\sI)^2 p_t^\sS(\varphi_t^\sS-\varphi_t^\sI)+\Big[\beta \lambda_t^{\sS*} \lambda_t^{\sI*} p_t^\sI\nonumber\\&\quad+\frac{\beta^2}{2}(\lambda_t^{\sI*})^2 (p_t^\sI)^2 (u_t^\sS-u_t^\sI)\Big] \psi_t^\sS-\gamma \psi_t^\sI,\nonumber \\
\dot{\psi}_t^\sR&=v^*_t \psi_t^\sS+\gamma \psi_t^\sI -\eta \psi_t^\sR,\nonumber \\
\dot{\chi}_t&=0,\label{eq:smfe_fbode}
\end{align}
\normalsize
with initial and terminal conditions
\small
$$u^e_T = 0, p^e_0 = \pi_0^e, B_0 = b_0, \quad e \in E,$$
$$\varphi^e_T = 0, \psi^e_0 = 0,\chi_T = 2K^b \min\{0,B_T\}, \quad e \in E.$$
\normalsize
\end{theorem}
\begin{proof}
    Due to the hierarchical structure of Stackelberg equilibrium, we consider the value function and density flow under MFNE together with budgeting process as a 7-dimension state of the principal, and solve a nonlinear optimal control problem to minimize $J_G$. For brevity, we omit detailed calculations and only state a high-level idea here.
    Let $\varphi$ be the adjoint process for $p$, $\psi$ be the adjoint process for $u$ and $\chi$ be that for $B$. Denote the running cost of principal's objective $J_G$ as $\ell$, the density flow dynamics~\eqref{eq:minor_KFP} as $f_p$ and that of value function~\eqref{eq:minor_HJB} as $f_u$. The principal's Hamiltonian writes
    \small
    \begin{equation}
        \label{eq:p_hamiltonian}
        \begin{aligned}
            &\mathcal{H}(t, p, u, B, \varphi, \psi, \chi ; \lambda, v)=\ell(t, p, \lambda, v)\\ &+\varphi \cdot f_p(t, p, u, \lambda, v) +\psi \cdot f_u(t, p, u, \lambda, v)+\chi b(t, \lambda, v, p).
        \end{aligned}
    \end{equation}
    \normalsize
    Let $\theta^*=\left(\lambda^*, v^*\right)$ be an SMFE control and $(p^*, u^*, B^*)$ be the corresponding state processes; by Pontryagin maximum principle, the state system and adjoint system under control optimality satisfy 
    \small
    $$\dot{p}^*=f_p^*, \dot{u}^*=f_u^*, \dot{B}^*=b^*; p^*(0)=\pi_0, u^*(T)=0, B^*(0)=b_0$$
    \normalsize
    and
    $$\begin{aligned}
    \dot{\varphi}_t&=-\partial_p \mathcal{H}(t, p_t^*, u_t^*, B_t^*, \varphi_t, \psi_t, \chi_t ; \lambda_t^*, v_t^*), \quad \varphi_T=0, \\
\dot{\psi}_t&=-\partial_u \mathcal{H}(t, p_t^*, u_t^*, B_t^*, \varphi_t, \psi_t, \chi_t ; \lambda_t^*, v_t^*), \quad \psi_0=0, \\
\dot{\chi}_t&=-\partial_B \mathcal{H}(t, p_t^*, u_t^*, B_t^*, \varphi_t, \psi_t, \chi_t ; \lambda_t^*, v_t^*), \\& \hspace{4.5cm} \chi_T=2 K^b \min \{0, B_T^*\},
\end{aligned}$$
where the principal control at time $t \in [0,T]$ satisfies
\small $$\left(\lambda_t^*, v_t^*\right) \in \argmin_{(\lambda, v) \in[0,1]^3 \times\left[0,V\right]} \mathcal{H}\left(t, p_t^*, u_t^*, B_t^*, \varphi_t, \psi_t, \chi_t ; \lambda, v\right).$$\normalsize
We are able to give a solution of the optimal controls using first-order condition. The claimed result then follows from straightforward calculations.
\end{proof}

\begin{lemma}
    \label{lemma:h_convexity}
    For any $(\varphi, \psi)$ that solves the FBODE system in~\eqref{eq:smfe_fbode}, under short enough horizon time $T > 0$, the Hamiltonian~\eqref{eq:p_hamiltonian} is strongly convex in $(\lambda^\sS,\lambda^\sI)$ and admits a unique minimizer as described by~\eqref{eq:smfe_control}.
\end{lemma}
\begin{proof}
    First note that by model formulation, all other variables are bounded: for any $t \in [0,T]$ and $e \in E$, $|B_t| \leq b_0+(m^\sS+m^\sI+m^\sR+m^v V)T =: \bar{B}$, $|\chi_t| \leq 2 K^b \bar{B}$; also $|p_t^e| \leq 1$ and by definition of value function, $u_t^e \leq c_\sI T = O(T)$.  Since the dynamics of $(\varphi,\psi)$ is linear with any admissible policy fixed, uniform boundedness is needed for the existence of the solution for ODEs related to $(\varphi,\psi)$. Denote the bounds as $\bar\varphi$ and $\bar\psi$ respectively, direct integration shows that $\bar\varphi = O(T)$ and $\bar\psi = O(T^2)$.

    To verify the strong convexity of Hamiltonian, we examine the invertibility of its Hessian matrix
    \small $$M_t:=\bigg(\begin{array}{cc}
\partial^2_{\lambda^{\sS}\lambda^{\sS}} \mathcal{H}_t & \partial^2_{\lambda^{\sS}\lambda^{\sI}} \mathcal{H}_t \\
\partial^2_{\lambda^{\sI}\lambda^{\sS}} \mathcal{H}_t & \partial^2_{\lambda^{\sI}\lambda^{\sI}} \mathcal{H}_t
\end{array}\bigg).$$ \normalsize
The four coefficients in the first two equations of~\eqref{eq:smfe_control} are precisely the entries of $M_t$.
    We thereby write the exact form of matrix $M_t$ as $(K^\sS \quad q_t; q_t \quad n_t)$, where
    \small
    \begin{equation*}
        \begin{aligned}
            q_t &:= \beta p_t^\sI p_t^\sS(\varphi_t^\sI-\varphi_t^\sS)+\beta p_t^\sI(u_t^\sS-u_t^\sI) \psi_t^\sS, \\
            n_t &:= K^\sI + \beta^2 (p_t^\sI)^2 p_t^\sS (u_t^\sS-u_t^\sI)(\varphi_t^\sI-\varphi_t^\sS) +\frac{1}{2}\beta^2 (p_t^\sI)^2(u_t^\sS-u_t^\sI)^2 \psi_t^\sS.
        \end{aligned}
    \end{equation*}
    \normalsize
    We have the estimation $|q_t| = O(T)$ and $|n_t - K^\sI| = O(T^2)$. Then the determinant of $M_t$ gives
    \small
    \begin{equation*}
        \begin{aligned}
            \det(M_t)&:=K^\sS n_t - (q_t)^2 \\ &= K^\sS K^\sI + K^\sS (n_t - K^\sI) - (q_t)^2 \\ &\geq K^\sS K^\sI - K^\sS |n_t - K^\sI| - |q_t|^2 \\ & = K^\sS K^\sI - O(T^2)
        \end{aligned}
    \end{equation*}
    \normalsize
    Select small enough $T$ such that $\det(M_t) \geq \frac{1}{2}K^\sS K^\sI>0$, the Hamiltonian admits a unique minimizer and first-order conditions apply.
\end{proof}

\begin{theorem}
    \label{thm:well_posedness}
    Under short enough time horizon $T > 0$, there exists a unique bounded continuous solution to the FBODE system~\eqref{eq:smfe_fbode}.
\end{theorem}

\begin{proof}
    The proof relies on Banach fixed point theorem and Lemma~\ref{lemma:h_convexity}. We construct a fixed-point mapping for the FBODE system in~\eqref{eq:smfe_fbode} and use Banach fixed point theorem to conclude the proof. Due to the mathematically heavy system, we deliver the high level idea first. Consider the policy space $\mathcal{M} := C([0,T];[0,1]^3 \times [0,V])$ and endow it with norm $\|\theta\|_T := \sup_{t\in[0,T]} \|\theta_t\|_2$. Fix $\theta \in \mathcal{M}$, we first solve the minor agent density flow and value function to obtain $(u,p)$, the existence and uniqueness of solution is guaranteed by Proposition~\ref{prop:minor_well_posedness}. Then with $(u,p)$, we solve the budget state process and its adjoint variable $(B, \chi)$. Next with $(u,p,B,\chi)$, we solve the adjoint processes $(\varphi,\psi)$ where the existence and uniqueness of a continuous solution is guaranteed due to the linear structure of the ODEs while other variables fixed. Finally, we solve the optimality conditions provided in~\eqref{eq:smfe_control} to achieve $\tilde{\theta}$. The described procedure corresponds to a mapping $\theta \mapsto h_1(\theta) =: (u,p) \mapsto h_2(u,p) =: (B,\chi) \mapsto h_3(B, \chi) =: (\varphi,\psi) \mapsto h_4(\varphi,\psi) =: \tilde{\theta}$.
    Consider the fixed-point mapping $f := h_4 \circ h_3 \circ h_2 \circ h_1$. Note $f(\mathcal{M}) \subset \mathcal{M}$ under large $V$, 
it suffices to show that $f$ is a contraction mapping. That is, for any $\theta_1, \theta_2 \in \mathcal{M}$, there is $\|f(\theta_1) - f(\theta_2)\|_T \leq C \|\theta_1 - \theta_2\|_T$ for some constant $C < 1$. A detailed derivation and further explanations can be found in the appendix.
\end{proof}

\section{Simulation Study}
\label{sec:simulation_study}
\subsection{Numerical Algorithm}
\label{sec:numerical_algo}
Following the theoretical characterization in Section~\ref{sec:smfe_charac}, we propose an algorithm to solve the FBODE system~\eqref{eq:smfe_fbode} characterizing the Stackelberg MFG equilibrium. The algorithm proceeds via fixed-point iteration, as summarized in Algorithm~\ref{alg:stackelberg_single_iter}. For a given candidate principal policy, we iteratively update: (i) the \textbf{minor} KFP equation for the state distribution and the \textbf{minor} HJB equation for the value function; (ii) the principal's budget trajectory via its forward ODE; and (iii) the principal's adjoint system, in which the adjoint variable associated with the minor state density evolves backward in time while that associated with the minor value function evolves forward. Using the updated state and adjoint trajectories, we compute a new principal policy from the optimality conditions~\eqref{eq:smfe_control} and repeat until convergence.

Regarding the short horizon assumption in Theorem~\ref{thm:well_posedness}, a possible numerical continuation strategy for longer horizons is to partition $[0,T]$ into sufficiently short subintervals and solve them recursively backward. The relation between the backward and forward variables computed on one subinterval is then used as the terminal condition for the preceding subinterval. The success of this procedure depends on the regularity of the resulting decoupling fields.

{
\small
\begin{algorithm}
\caption{\small Fixed-point iteration for SMFE \label{alg:stackelberg_single_iter}}
\footnotesize
\textbf{Input:} Model parameters. Initialize principal control paths $\btheta^{(0)}=(
\lambda_t^{\sS,(0)},
\lambda_t^{\sI,(0)},
\lambda_t^{\sR,(0)},
v_t^{(0)})_{t\in[0,T]},$
the state trajectories
$\bp^{(0)}=(p_0^{(0)},p_{\Delta t}^{(0)},\dots,p_T^{(0)}),
\bu^{(0)}=(u_0^{(0)},u_{\Delta t}^{(0)},\dots,u_T^{(0)})$,
the budget path $\boldsymbol B^{(0)}=(B_0^{(0)},B_{\Delta t}^{(0)},\dots,B_T^{(0)})$,
and the adjoint trajectories
$\bvarphi^{(0)}=(\varphi_0^{(0)},\varphi_{\Delta t}^{(0)},\dots,\varphi_T^{(0)})$, $
\bpsi^{(0)}=(\psi_0^{(0)},\psi_{\Delta t}^{(0)},\dots,\psi_T^{(0)})$.
Set iteration counter $m=0$ and tolerance $\epsilon$.
\vskip1mm
\begin{algorithmic}[1]

\WHILE{
$\|x^{(m+1)}-x^{(m)}\|>\epsilon$ for any $x \in \{\blambda, \bp, \bu, \bvarphi, \bpsi, \boldsymbol{B}\}$}\vskip1mm
    \STATE Calculate the follower best-response socialization level
    $\alpha_t^{\sS,(m+1)}=
    \lambda_t^{\sS,(m)}
    +\frac{\beta}{2}\lambda_t^{\sI,(m)}p_t^{\sI,(m)}
    (u_t^{\sS,(m)}-u_t^{\sI,(m)})$,
    $\alpha_t^{\sI,(m+1)}=\lambda_t^{\sI,(m)}$, and 
    $\alpha_t^{\sR,(m+1)}=\lambda_t^{\sR,(m)}$.\vskip1mm

    \STATE Update $\bp^{(m+1)}$ by solving KFP with $\boldsymbol\alpha^{(m+1)}=(\alpha^{\sS,(m+1)},\alpha^{\sI,(m+1)},\alpha^{\sR,(m+1)})$ and $\btheta^{(m)}$.\vskip1mm

    \STATE Update $\bu^{(m+1)}$ by solving HJB with $\boldsymbol\alpha^{(m+1)},\btheta^{(m)},\bp^{(m+1)}.$\vskip1mm

    \STATE Update $\boldsymbol{B}^{(m+1)}$ by solving ODE~\eqref{eq:budget_spec} with $\bp^{(m+1)}$ and $\btheta^{(m)}$.\vskip1mm

    \STATE Set $\chi^{(m+1)} = 2K^b \min\{0,B_T^{(m+1)}\}$
    and $\chi_t^{(m+1)}\equiv \chi^{(m+1)}$ for all $t\in[0,T]$.\vskip1mm

    \STATE Update $\bvarphi^{(m+1)}$ by solving adjoint ODE in~\eqref{eq:smfe_fbode}
    with $\bp^{(m+1)}$, $\bu^{(m+1)}$, $\bpsi^{(m)}$, $\boldsymbol{\alpha}^{(m+1)}$, $\btheta^{(m)}$, and $\chi^{(m+1)}$.\vskip1mm

    \STATE Update $\bpsi^{(m+1)}$ by solving adjoint ODE in~\eqref{eq:smfe_fbode}
    with $\bp^{(m+1)}$, $\bu^{(m+1)}$, $\bvarphi^{(m+1)}$, and $\btheta^{(m)}$.\vskip1mm

    \STATE Compute $\btheta^{(m+1)}$ by plugging
    $\bp^{(m+1)},
    \bu^{(m+1)},
    \bvarphi^{(m+1)},
    \bpsi^{(m+1)},
    \chi^{(m+1)}$
    into optimality conditions in~\eqref{eq:smfe_control}.\vskip1mm
    \STATE Set $m\leftarrow m+1$.\vskip1mm

\ENDWHILE

\STATE Perform one final forward-backward update under the converged control $\hat{\btheta}$ to obtain
$\hat{\boldsymbol\alpha},
\hat{\bp},
\hat{\bu},
\hat{\boldsymbol B},
\hat{\bvarphi},
\hat{\bpsi},
\hat{\chi}$.\vskip1mm

\RETURN
$(\hat{\btheta}, \hat{\boldsymbol\alpha},
\hat{\bp},
\hat{\bu},
\hat{\boldsymbol B},
\hat{\bvarphi},
\hat{\bpsi},
\hat{\chi})$
\end{algorithmic}
\end{algorithm}
}

\subsection{Experiment Results}
We present the numerical results under parameter setting as specified in Table~\ref{tab:numerical_param} and compare the Stackelberg MFG equilibrium outcome to two benchmark cases in which the government policy is fixed. In the first benchmark, there is essentially no regulation with relaxed social distancing guidelines and no vaccination, where $\lambda_t^\sS=\lambda_t^\sI=\lambda_t^\sR=0.95$ and $v_t=0$. In the second benchmark, the government adopts a light regulation policy that exactly matches the public health authority's recommended levels for social distancing and vaccination, namely $\lambda_t^e=\bar{\lambda}^e$ and $v_t=\bar v$. In the figures, these three cases are labeled \texttt{Optim}, \texttt{NoReg}, and \texttt{Light}, respectively.

Since the FBODE system in~\eqref{eq:smfe_fbode} is mathematically heavy, we first illustrate the convergence of Algorithm~\ref{alg:stackelberg_single_iter} in the top panel of Figure~\ref{fig:conv_density}. We briefly comment here on the numerical convergence behavior observed under other parameter configurations. The convergence is stable for short time horizons, in agreement with Theorem~\ref{thm:well_posedness}. It also remains stable when the scaling parameters in principal's cost functional are relatively small. We note that issues may occur when the initial budget $b_0$ is very small or the vaccination cost is very high, as the resulting optimal policy may fall outside the admissible set without additional clipping, which is not considered in this paper.
We chose the model parameters such that the solutions fall into the admissible set.

In the bottom panels of Figure~\ref{fig:conv_density}, we compare the state density flows under the Stackelberg MFG equilibrium and the two fixed-policy benchmarks. We observe that infection is well-controlled in the SMFE (denoted with \texttt{Optim}) case and remains visibly lower than under either fixed policies. Figure~\ref{fig:budget_allo} reports the budget-related results. We first note that the unit-agent spending (obtained by dividing the cost by the state density to which the policy is applied) associated with social distancing for infected individuals is the highest among all policy instruments, which is consistent with the practical observation that quarantining infected individuals is particularly costly due to, for example, accommodation and sanitization expenses. Compared with the light-regulation benchmark, the Stackelberg solution allocates a larger share of expenditure to vaccination and in general less to social distancing. As a result, the remaining budget in the SMFE case decreases more sharply at the beginning, reflecting the larger initial investment in vaccination at early stages, and at the end induces a substantially smaller budget deficit. This shows that it is better for the government to take actions early to prevent infections proactively.

Figure~\ref{fig:policy_minor} presents the SMFE policies and also shows the corresponding MFG equilibrium socialization levels of the minor agents. We observe that the vaccination level decreases over time and eventually falls below the public health authority's recommended level, which is consistent with the budget allocation results. We also observe that the social distancing guideline for state $\sS$ gradually declines towards the recommended level $0.8$, while susceptible individuals further reduce their socialization level when the infection prevalence is high to proactively protect themselves from the disease. 

We remark that the numerical results depend strongly on the model parameters. For the disease-specific parameters, we use values inspired by~\cite{liu2025incorporatingauthorityperceptioneconomic}, for example $\gamma$ is taken to be $0.08$ to represent around 1.8 week of recovery time.
{\small
\begin{table}[h]
\centering
\begin{tabular}{p{2cm}p{2.15cm}p{3.2cm}}
\hline
\textbf{Parameter} & \textbf{Value} & \textbf{Description} \\
\hline
$T$ & $70$ & Time horizon \\
$\beta$ & $0.5$ & Transmission rate \\
$\gamma$ & $0.08$ & Recovery rate \\
$\eta$ & $0.0$ & Waning of immunity rate \\
$c_\sI$ & $1.0$ & Infection cost \\
$K^p$ & $0.001$ & Cost weight for infection \\
$(K^\sS, K^\sI, K^\sR, K^v)$ & $(0.1, 0.1, 0.1, 0.3)$ & Cost weights for policies \\
$K^b$ & $1.0$ & Cost weight for budget \\
$(m^\sS, m^\sI, m^\sR, m^v)$ & $(0.1, 0.2, 0.1, 0.5)$ & Budget dynamics scalings \\
$(\bar\lambda^\sS, \bar\lambda^\sI, \bar\lambda^\sR, \bar{v})$ & $(0.8, 0.6, 0.95, 0.02)$ & Recommended policy levels \\
$\pi_0$ & $(0.95, 0.05, 0.0)$ & Initial state densities \\
$b_0$ & $1.0$ & Initial budget \\
$N$ & $10000$ & Number of time steps \\
$\epsilon$ & $0.001$ & Convergence tolerance \\
\hline
\end{tabular}
\caption{\small Model Parameters for the Experiments}
\label{tab:numerical_param}
\end{table}}

\begin{figure}
    \centering
    \includegraphics[width=1.\linewidth]{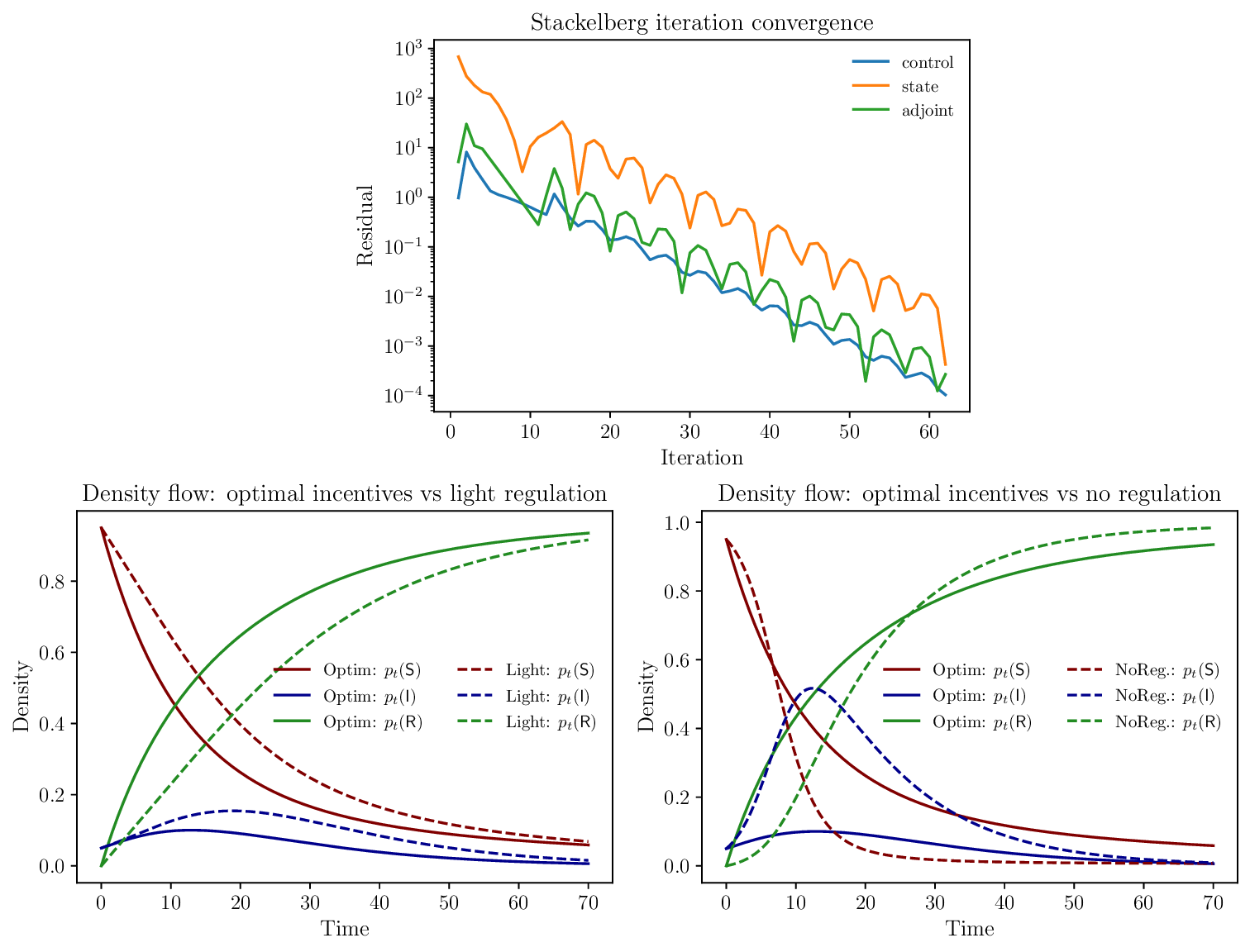}
    \caption{\textbf{Top}: algorithm convergence verification. \textbf{Bottom}: state density flow pairwise comparisons between SMFE (\texttt{Optim}) and fixed policies (\texttt{Light} and \texttt{NoReg}).}
    \label{fig:conv_density}
\end{figure}

\begin{figure}
    \centering
    \includegraphics[width=1.\linewidth]{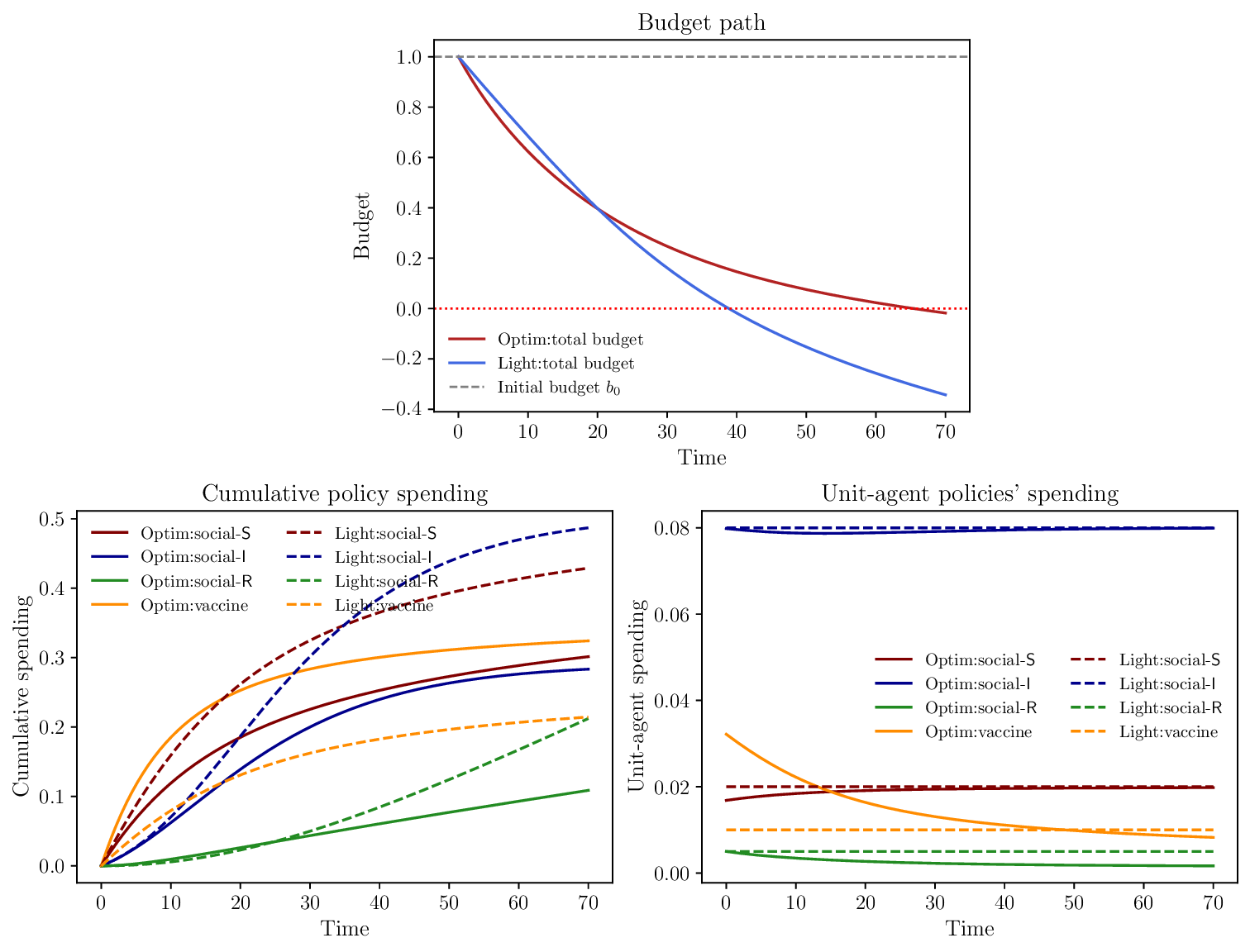}
    \caption{\textbf{Top}: The total budget levels. \textbf{Bottom}: Comparison of cumulative spending (left) and unit-agent spending (right) by policy items.}
    \label{fig:budget_allo}
\end{figure}

\begin{figure}
    \centering
    \includegraphics[width=0.98\linewidth]{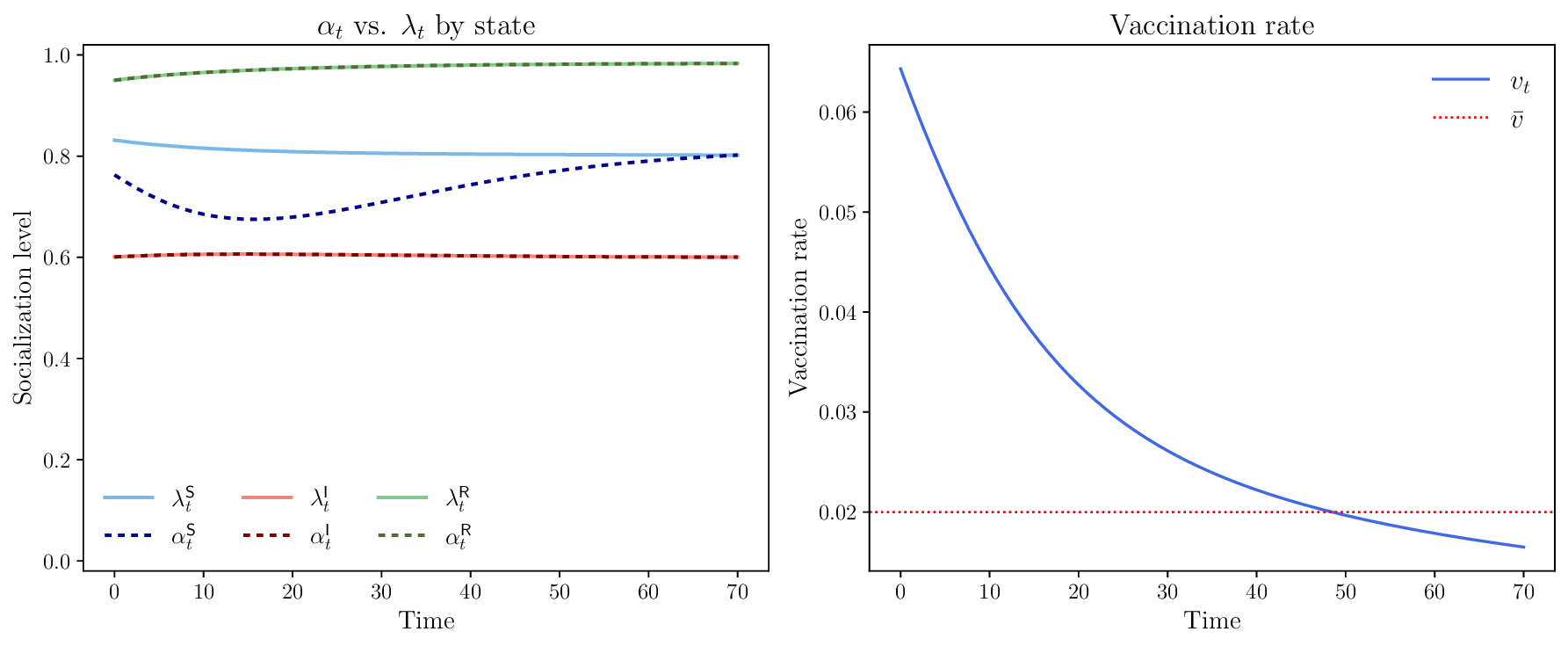}
    \caption{Social distancing and vaccination policies of the government and the minor agent's corresponding MFG socialization levels, under SMFE.}
    \label{fig:policy_minor}
\end{figure}

\section{Conclusions and Future Directions}
\label{sec:future_dir}
This paper introduces the first continuous-time Stackelberg MFGs with additional state processes for the principal with a motivation to model optimal epidemic control with human response. We give necessary conditions for the Stackelberg mean field equilibrium and establish its well-posedness.  We further present numerical results. Several interesting directions remain open for future work.

First, regarding model formulation, we simplify individual vaccination uptake by treating it as a decision made directly by the principal, in order to focus on the allocation of budget between social distancing and vaccination. A natural extension is to incorporate individual decision-making on vaccination. Second, the budget state process in the current framework is deterministic. Introducing noise into the budget dynamics may capture the uncertainty that arises from public health resource allocation and policy implementation in practice, we plan to address this situation in the general model form. Third, this framework could be generalized to multi-population MFGs or to finite-state graphon games to capture the heterogeneities in the population.

\bibliographystyle{ieeetr}
\bibliography{reference.bib}

\newpage
\onecolumn

\appendices

\section{Proof of Theorem~\ref{thm:well_posedness}}
We construct a fixed-point mapping for the FBODE system in~\eqref{eq:smfe_fbode} and use Banach fixed point theorem to conclude the proof. Due to the heavy system, we deliver a high level idea first. Consider the policy space $\mathcal{M} := C([0,T];[0,1]^3 \times [0,V])$ and deploy it with norm $\|\theta\|_T := \sup_t \|\theta_t\|_2$. Fix $\theta \in \mathcal{M}$, we first solve the minor agent density flow and value function to obtain $(u,p)$, the existence and uniqueness of solution is guaranteed by Proposition~\ref{prop:minor_well_posedness}. Then with $(u,p)$, we solve the budget state process and its adjoint variable $(B, \chi)$. Next with $(u,p,B,\chi)$, we solve the adjoint processes $(\varphi,\psi)$ where the existence and uniqueness of a continuous solution is guaranteed since it's a linear ODE with other variables fixed. Finally, we solve the optimality conditions provided in~\eqref{eq:smfe_control} to achieve $\tilde{\theta}$. The described procedure corresponds to a mapping $\theta \mapsto h_1(\theta) =: (u,p) \mapsto h_2(u,p) =: (B,\chi) \mapsto h_3(B, \chi, u, p, \theta) =: (\varphi,\psi) \mapsto h_4(\varphi,\psi, u, p, \chi) =: \tilde{\theta}$. Consider the fixed-point mapping $f := h_4 \circ h_3 \circ h_2 \circ h_1$. Note $f(\mathcal{M}) \subset \mathcal{M}$ under large $V$, 
it suffices to show that $f$ is a contraction mapping. That is, for any $\theta_1, \theta_2 \in \mathcal{M}$, there is $\|f(\theta_1) - f(\theta_2)\|_T \leq C \|\theta_1 - \theta_2\|_T$ for some constant $C < 1$.

Consider $\theta_1 = (\lambda_1,v_1)$ and $\theta_2 = (\lambda_2,v_2)$, and denote the corresponding trajectories as $(u^1,p^1,B^1,\varphi^1,\psi^1)$ and $(u^2,p^2,B^2,\varphi^2,\psi^2)$. For notational simplicity, we also write the induced value differences as $
\Delta p_t^e:=p_t^{1,e}-p_t^{2,e},
\Delta u_t^e:=u_t^{1,e}-u_t^{2,e},
\Delta Z_t:=Z_t^1-Z_t^2$, $\Delta \varphi_t^e =\varphi_t^{1,e} - \varphi_t^{2,e}$, $\Delta \psi_t^e =\psi_t^{1,e} - \psi_t^{2,e}$, $\Delta \hat\alpha^e_t:=\hat\alpha_t^{1,e}-\hat\alpha_t^{2,e}$, 
$\Delta \lambda_t^e:=\lambda_{t}^{1,e}-\lambda_{t}^{2,e},
\Delta v_{t}:=v_{t}^1-v_{t}^2$.

We clarify some boundedness tools first. For the mean-field term, we have $|\Delta Z_t|\leq |\Delta p_t^\sI|+ |\Delta \lambda_t^{\sI}|$. Also by definition of value function ,we have $\left|u_t(e)\right| \leq c_\sI T =: \bar{U}$, for any $e \in E$, $t \in [0,T]$. For the MFNE socialization level, we have $|\Delta\hat\alpha_t^\sS|\le|\Delta \lambda_t^{\sS}|+\beta\bar U |\Delta p_t^\sI|+\beta\bar U |\Delta \lambda_t^{\sI}|+\frac{\beta}{2}|\Delta u_t^\sS|+\frac{\beta}{2}|\Delta u_t^\sI|$. We then proceed with the fixed-point mapping.

\textbf{Step 1. Estimation on $h_1$.} For backward dynamics for minor agents, we have
\begin{equation*}
    \begin{aligned}
        \frac{d}{dt}|\Delta u_t^\sS|^2
        &=2\Delta u_t^\sS\Delta\dot{u}_t^\sS \\
        &\leq 2|\Delta u_t^\sS|\Big[2\beta\bar U|\Delta\lambda_t^{\sS}|+2\beta\bar U|\Delta Z_t|+\beta|\Delta u_t^\sS|
        +\beta|\Delta u_t^\sI|+V|\Delta u_t^\sS|+V|\Delta u_t^\sR| + 2\bar{U} |\Delta v_t| +2|\Delta \hat{\alpha}_t^\sS|\Big] \\
        &\leq 2|\Delta u_t^\sS|\Big[a_1|\Delta u_t^\sS|+a_2|\Delta u_t^\sI|+a_3|\Delta u_t^\sR| +a_4|\Delta p_t^\sI|+a_5|\Delta v_t|+a_6|\Delta \lambda_t^{\sS}|+a_7|\Delta \lambda_t^{\sI}|\Big] \\
        &\leq (2a_1+a_2+\cdots+a_7)|\Delta u_t^\sS|^2+a_2|\Delta u_t^\sI|^2 +a_3|\Delta u_t^\sR|^2+a_4|\Delta p_t^\sI|^2+a_5|\Delta v_t|^2+a_6|\Delta \lambda_t^{\sS}|^2+a_7|\Delta \lambda_t^{\sI}|^2.
    \end{aligned}
\end{equation*}
Here
$a_1=2\beta + V$,
$a_2=2\beta$, $a_3=V$,
$a_4=4 \beta \bar{U}$,
$a_5=2\bar{U}$, $a_6=2\beta \bar U+2$,
$a_7=4 \beta \bar{U}$.
Similarly for the value function of state $\sI$ and $\sR$, there is
$
\frac{d}{dt}|\Delta u_t^\sI|^2\leq 3\gamma|\Delta u_t^\sI|^2+\gamma|\Delta u_t^\sR|^2$ and $\frac{d}{dt}|\Delta u_t^\sR|^2\le3\eta|\Delta u_t^\sR|^2+\eta|\Delta u_t^\sS|^2.
$
We combine to have
$$
\frac{d}{dt}\|\Delta u_t\|^2\le A_1\|\Delta u_t\|^2+A_2\|\Delta p_t\|^2+A_3\|\Delta \theta_t\|^2,
$$
where
$A_1:=\max\{2a_1+a_2+a_3+a_4+a_5+a_6+a_7+\eta,\,a_2+3\gamma,a_3+\gamma+3\eta\}$, $A_2:=a_4$ and $A_3 := \max\{a_5, a_6, a_7\}$.
Since $\Delta u_T=0$, we apply Gr\"onwall's inequality to have
\begin{equation*}
        \|\Delta u_t\|^2
        \leq\int_t^T e^{A_1(t-s)}\Big[A_2\|\Delta p_s\|^2+A_3\|\Delta \theta_s\|^2\Big]ds.
\end{equation*}
Therefore, taking the supremum over $t\in[0,T]$ yields $\Delta U:=\sup_{t\in[0,T]}\|\Delta u_t\|^2
\leq e^{A_1T}T\big[A_2 \Delta P+A_3 \sup_t \|\Delta \theta_t\|^2\big]$, where $\Delta P:=\sup_{t\in[0,T]}\|\Delta p_t\|^2$. For the minor agent KFP equations, we similarly have, for some constants $b_1,\dots,b_6>0$ depending only on the uniform bounds,
\begin{equation*}
    \begin{aligned}
        \frac{d}{dt}|\Delta p_t^\sS|^2
        &=2\Delta p_t^\sS{\Delta \dot{p}}_t^\sS \\
        &\leq 2|\Delta p_t^\sS|
        \Big[
        \beta |\Delta \hat{\alpha}_t^\sS|
        +\beta |\Delta p_t^\sI|
        +\beta |\Delta \lambda_t^\sI|
        +\beta |\Delta p_t^\sS|
        +|\Delta v_t|
        +V |\Delta p_t^\sS|
        +\eta |\Delta p_t^\sR|
        \Big] \\
        &\leq 2|\Delta p_t^\sS|
\Big[
b_1 |\Delta u_t^\sS|
+b_2|\Delta u_t^\sI|
+(b_3+V)|\Delta p_t^\sS|
+b_4|\Delta p_t^\sI|+\eta |\Delta p_t^\sR|
+b_5|\Delta \lambda_t^{\sS}|
+b_6|\Delta \lambda_t^{\sI}|
+|\Delta v_{t}|
\Big] \\
& \le
(b_1+\cdots+b_6+b_3+V+\eta+1)|\Delta p_t^\sS|^2+b_1|\Delta u_t^\sS|^2
+b_2|\Delta u_t^\sI|^2
+b_4|\Delta p_t^\sI|^2
+\eta|\Delta p_t^\sR|^2
+b_5|\Delta \lambda_t^{\sS}|^2
+b_6|\Delta \lambda_t^{\sI}|^2
+|\Delta v_t|^2.
    \end{aligned}
\end{equation*}
For state $\sI$ and $\sR$, we similarly have
$$
\frac{d}{dt}|\Delta p_t^\sI|^2\leq 
\big(b_1+b_2+b_3+2b_4+2\gamma+b_5+b_6\big)|\Delta p_t^\sI|^2 + b_1|\Delta u_t^\sS|^2 + b_2|\Delta u_t^\sI|^2 + b_3|\Delta p_t^\sS|^2 + b_5|\Delta \lambda_t^{\sS}|^2 + b_6|\Delta \lambda_t^{\sI}|^2,
$$
and
$\frac{d}{dt}|\Delta p_t^\sR|^2
\leq V|\Delta p_t^\sS|^2 + \gamma|\Delta p_t^\sI|^2 +(1 + V + \gamma + 2\eta)|\Delta p_t^\sR|^2 + |\Delta v_t|^2.$
Combining the three estimates yields
$$
\frac{d}{dt}\|\Delta p_t\|^2
\leq B_1\|\Delta p_t\|^2+B_2\|\Delta u_t\|^2+B_3 \|\Delta \theta_t\|^2,
$$
where $B_1:=\max\{b_1 +\cdots+b_6+2b_3+2V+\eta+1,b_1 +\cdots+b_6+2b_4+3\gamma,1+V+\gamma+3\eta\}$,
$B_2:=\max\{2b_1,2b_2\}$ and $B_3 = \max \{2b_5,2b_6,2\}$. Since $\Delta p_0=0$, we again use Gr\"onwall's inequality to have
$$
\|\Delta p_t\|^2
\le \int_0^t e^{B_1(t-s)}
\big[
B_2\|\Delta u_s\|^2+B_3 \|\Delta \theta_s\|^2
\big]ds.
$$
Taking the supremum over $t\in[0,T]$ gives
$
\Delta P
\leq e^{B_1T}T\big[B_2 \Delta U+B_3 \sup_t \|\Delta \theta_t\|^2\big]
$. Plug in the backward estimate gives
$$\Delta P \leq e^{(A_1+B_1)T}A_2 B_2 T^2 \Delta P + [e^{(A_1+B_1)T}A_3 B_2 T^2 +e^{B_1 T} B_3 T]\sup_t \|\Delta \theta_t\|^2$$
For short enough time horizon such that $e^{(A_1+B_1)T}A_2 B_2 T^2 < 1$,
then we have
$$
\Delta P
\leq
\frac{e^{(A_1+B_1)T}A_3 B_2 T^2 +e^{B_1 T} B_3 T}
{1-e^{(A_1+B_1)T}A_2 B_2 T^2}  \sup_t \|\Delta \theta_t\|^2.
$$
A similar but symmetric procedure holds for $\Delta U$, we therefore have $\|\Delta p\|_T \leq C_p \|\Delta \theta\|_T$ and $\|\Delta u\|_T \leq C_u \|\Delta \theta\|_T$ for some $C_p > 0$ and $C_u > 0$.

\textbf{Step 2. Estimation on $h_2$.} The estimate on budget process and its adjoint variable is not heavy, observing that the right-hand side of its dynamics is a product sum of variables. We follow similar path as step 1 and obtain
$$\sup_t \|\Delta B_t\|^2 \leq e^{C_1 T}C_2 T \sup_t\|\Delta p_t\|^2 + e^{C_1 T}C_3 T\sup_t \|\Delta \theta_t\|^2,$$
where $C_1 := 7m^e + m^v V+m^v$, $C_2 := \max\{2m^\sS+m^vV,2m^\sI,2m^\sR\}$ and $C_3:=\max\{m^\sS, m^\sI, m^\sR,m^v\}$. And trivially there is $\sup_t \|\Delta \chi_t\|^2 \leq |2 K^b \Delta B_T|^2\leq 4(K^b)^2 \sup_t \|\Delta B_t\|^2$. We have $\|\Delta B\|_T \leq C_b \|\Delta \theta\|_T$ and $\|\Delta \chi\|_T \leq C_\chi \|\Delta \theta\|_T$.

\textbf{Step 3. Estimation on $h_3$.}
Since the budget state process $B$ does not enter the adjoint system explicitly, it suffices to estimate the dependence of $(\phi,\psi)$ on $(u,p,\chi,\theta)$. By Lemma~\ref{lemma:h_convexity}, there exist constants $\bar \Phi,\bar \Psi > 0$ such that $|\phi_t^e| \le \bar \Phi$ and $|\psi_t^e| \le \bar \Psi$ for all $e \in \{\sS,\sI,\sR\}$ and $t \in [0,T]$.
For the backward adjoint equations, we have the estimation
$$
\begin{aligned}
\frac{d}{dt} |\Delta \varphi_t^{\sS}|^2
&= 2\Delta \varphi_t^{\sS}\Delta \dot \varphi_t^{\sS} \\ & \leq 2|\Delta \varphi_t^{\sS}|
\Bigl[
d_1 |\Delta \varphi_t^{\sS}| + d_2 |\Delta \varphi_t^{\sI}| + d_3 |\Delta \varphi_t^{\sR}| + d_4 |\Delta p_t^\sI| +d_5 |\Delta u_t^\sS| + d_6 |\Delta u_t^\sI| + d_7 |\Delta \chi_t| + d_8 |\Delta \lambda_t^\sS| + d_9 |\Delta \lambda_t^\sI| + d_{10} |\Delta v_t|
\Bigr] \\
&\le (2d_1+d_2+\cdots+d_{10})|\Delta \varphi_t^{\sS}|^2 + d_2|\Delta \varphi_t^{\sI}|^2 + d_3|\Delta \varphi_t^{\sR}|^2 + d_4\|\Delta p_t\|^2 + \tilde{d}_5\|\Delta u_t\|^2 + \tilde{d}_6\|\Delta \theta_t\|^2 + d_7|\Delta \chi_t|^2,
\end{aligned}$$
where $d_1 := 2\beta + V$, $d_2 := 2\beta$, $d_3 := V$, $d_4 := 2\beta^2 \bar{\varphi} \bar{U}+2\beta \bar{\varphi}$, $d_5 := \frac{1}{2}\beta$, $d_6 := \frac{1}{2}\beta$, $d_7 := 2m^\sS + m^v V$, $d_8 := 2\beta \bar{\varphi} + m^\sS \bar{\chi}$, $d_9 := 2\beta^2 \bar{U} \bar{\varphi} + 2 \beta \bar{\varphi}$, $d_{10} := 2\bar{\varphi} + m^v \bar{\chi}$; and $\tilde{d}_5:=\max\{d_5,d_6\}$, $\tilde{d}_6:=\max\{d_8,d_9,d_{10}\}$.
$$
\begin{aligned}
\frac{d}{dt} |\Delta \varphi_t^{\sI}|^2
&= 2\Delta \varphi_t^{\sI}\Delta \dot \varphi_t^{\sI} \\
&\leq 2|\Delta \varphi_t^{\sI}|
\Bigl[
d_{11} |\Delta \varphi_t^{\sS}| + d_{12} |\Delta \varphi_t^{\sI}| + d_{13} |\Delta \varphi_t^{\sR}| + d_{14} |\Delta \psi_t^{\sS}| + d_{15}|\Delta p_t^\sS| + d_{16}|\Delta p_t^\sI| \\ &\quad\quad  +d_{17}|\Delta u_t^\sS| + d_{18}|\Delta u_t^\sI| + d_{19}|\Delta \chi_t| + d_{20}|\Delta \lambda_t^\sS| + d_{21}|\Delta \lambda_t^\sI|
\Bigr] \\
&\leq d_{12}|\Delta \varphi_t^{\sS}|^2 + (d_{11}+2d_{12}+\cdots+d_{21})|\Delta \varphi_t^{\sI}|^2 + d_{13}|\Delta \varphi_t^{\sR}|^2 + d_{14}|\Delta \psi_t^{\sS}|^2 + \tilde{d}_{15}\|\Delta p_t\|^2 \\ & \quad \quad + \tilde{d}_{16}\|\Delta u_t\|^2 + d_{19}|\Delta \chi_t|^2 + \tilde{d}_{17}\|\Delta \theta_t\|^2,
\end{aligned}
$$
where $d_{11}:=\beta + 2\beta^2\bar{U}$, $d_{12}:=\beta + 2\beta^2\bar{U}+\gamma$, $d_{13} := \gamma$, $d_{14}:=2\beta\bar{U}+2\beta^2\bar{U}^2$, $d_{15}:=2\beta\bar{\varphi}+4\beta^2 \bar{\varphi}^2 \bar{U}^2$, $d_{16}:=2K^p + 4\beta^2\bar{\varphi}^2\bar{U}^2+2\beta^2\bar{U}^2\bar{\psi}$, $d_{17}:=2\beta^2\bar{\varphi}^2+\beta^2\bar{\psi}\bar{U}+\beta \bar{\psi}$, $d_{18}:=2\beta^2\bar{\varphi}^2+\beta^2\bar{\varphi}\bar{U}+\beta \bar{\psi}$, $d_{19}:=2m^\sI$, $d_{20}:=2\beta \bar{\varphi}+2\beta \bar{U}\bar{\psi}$, $d_{21}:=2 \beta \bar{\varphi}+8 \beta^2 \bar{\varphi}^2 \bar{U}^2+2 \beta \bar{U} \bar{\psi}+4 \beta^2 \bar{U}^2 \bar{\psi}+m^\sI \bar{\chi}$; and $\tilde{d}_{15}:=\max\{d_{15}, d_{16}\}$, $\tilde{d}_{16}:=\max\{d_{17}, d_{18}\}$, $\tilde{d}_{17}:=\max\{d_{20}, d_{21}\}$.
$$
\begin{aligned}
\frac{d}{dt} |\Delta \varphi_t^{\sR}|^2
= 2\Delta \varphi_t^{\sR}\Delta \dot \varphi_t^{\sR} &\leq 2|\Delta \varphi_t^{\sR}|
\big[
\eta |\Delta \phi_t^{\sS}| + \eta |\Delta \phi_t^{\sR}| + 2 m^\sR|\Delta \chi_t| + m^\sR \bar{\chi}\ |\Delta \lambda_t^\sR|
\big] \\
&\le \eta |\Delta \varphi_t^{\sS}|^2 + (3\eta + 2 m^\sR + m^\sR \bar{\chi})|\Delta \varphi_t^{\sR}|^2 + 2 m^\sR|\Delta \chi_t|^2 + m^\sR \bar{\chi}\ |\Delta \lambda_t^\sR|^2.
\end{aligned}
$$
Combining the three estimates yields
$$
\frac{d}{dt}\|\Delta \varphi_t\|_2^2 \leq D_1\|\Delta \varphi_t\|^2 + D_2\|\Delta \psi_t\|^2 + D_3\|\Delta p_t\|^2 + D_4\|\Delta u_t\|^2 + D_5|\Delta \chi_t|^2 + D_6\|\Delta \theta_t\|^2,
$$
where $D_1:=\max\big\{
  2d_1 + d_2 + \cdots + d_{10},
  d_{11} + 2d_{12} + \cdots + d_{21},
  3\eta + 2m^{\sR} + m^{\sR}\bar{\chi}\big\}$, $D_2 := d_{14}$, $D_3:= d_4+\tilde{d}_{15}$, $D_4 := \tilde{d}_5+ \tilde{d}_{16}$, $D_5:=\max\{ d_7,d_{19},2m^{\sR}\}$, $D_6 := \max\{ \tilde{d}_6,\tilde{d}_{17},m^{\sR}\bar{\chi}\}$.
Since $\Delta \varphi_T = 0$, Gr\"onwall's inequality gives
$$
\|\Delta \varphi_t\|^2 \le \int_t^T e^{D_1(t-s)}
\Bigl[
D_2\|\Delta \psi_s\|^2 + D_3\|\Delta p_s\|^2 + D_4\|\Delta u_s\|^2 + D_5|\Delta \chi_s|^2 + D_6\|\Delta \theta_s\|^2
\Bigr] ds.
$$
Hence, taking the supremum over $t \in [0,T]$ gives $\Delta \Phi \le e^{D_1T}T\big[D_2\Delta \Psi + D_3\Delta P + D_4\Delta U + D_5\Delta \chi + D_6 \sup_{t\in[0,T]}\|\Delta \theta_t\|^2 \big]$. Here $\Delta \Phi := \sup_{t\in[0,T]}\|\Delta \varphi_t\|^2$, $\Delta \Psi := \sup_{t\in[0,T]}\|\Delta \psi_t\|^2$ and $\Delta \chi := |\Delta \chi|^2$.
For the forward adjoint equations, we similarly have estimation
$$
\begin{aligned}
\frac{d}{dt} |\Delta \psi_t^{\sS}|^2
&= 2\Delta \psi_t^{\sS}\Delta \dot \psi_t^{\sS} \\
&\le 2|\Delta \psi_t^{\sS}|
\Bigl[
e_1|\Delta \psi_t^{\sS}| + e_2|\Delta \psi_t^{\sR}| + e_3|\Delta \varphi_t^{\sS}| + e_4|\Delta \varphi_t^{\sI}| + e_5|\Delta p_t^\sS| + e_6|\Delta p_t^\sI| \\ & \quad \quad + e_7|\Delta u_t^\sS| + e_8|\Delta u_t^\sI| + e_9|\Delta \lambda^\sS_t| + e_{10}|\Delta \lambda^\sI_t| + e_{11}|\Delta v_t|
\Bigr] \\
&\leq (2e_1+e_2+\cdots+e_{11})|\Delta \psi_t^{\sS}|^2 + e_2|\Delta \psi_t^{\sR}|^2 + e_3|\Delta \varphi_t^{\sS}|^2 + e_4|\Delta \varphi_t^{\sI}|^2 + \tilde{e}_5\|\Delta p_t\|^2 + \tilde{e}_6\|\Delta u_t\|^2 + \tilde{e}_7\|\Delta \theta_t\|^2,
\end{aligned}
$$
where $e_1:=\beta + \beta^2\bar{U} + V$, $e_2 := \eta$, $e_3:=\frac{1}{2}\beta^2$, $e_4:=\frac{1}{2}\beta^2$, $e_5:=\beta^2 \bar{\varphi}$, $e_6:=2\beta^2\bar{\varphi}+2\beta^2\bar{\psi}\bar{U}+\beta \bar{\psi}$, $e_7:=\frac{1}{2}\beta^2 \bar{\psi}$, $e_8:=\frac{1}{2}\beta^2 \bar{\psi}$, $e_9:=\beta\bar{\psi}$, $e_{10}:=2\beta^2\bar{\varphi}+2\beta^2\bar{\psi}\bar{U}+\beta \bar{\psi}$, $e_{11}:=\bar{\psi}$; and $\tilde{e}_{5}:=\max\{e_5,e_6\}$, 
$\tilde{e}_{6}:=\max\{e_7,e_8\}$, $\tilde{e}_{7}:=\max\{e_9,e_{10},e_{11}\}$. We have for state $\sI$
$$
\begin{aligned}
\frac{d}{dt} |\Delta \psi_t^{\sI}|^2
&= 2\Delta \psi_t^{\sI}\Delta \dot \psi_t^{\sI} \\
&\leq 2|\Delta \psi_t^{\sI}|
\Bigl[
e_{12}|\Delta \psi_t^{\sS}| + e_{13}|\Delta \psi_t^{\sI}| + e_{14}|\Delta \varphi_t^{\sS}| + e_{15}|\Delta \varphi_t^{\sI}| + e_{16}|\Delta p_t^\sS| + e_{17}|\Delta p_t^\sI| \\ & \quad\quad + e_{18}|\Delta u_t^\sS| + e_{19}|\Delta u_t^\sI| + e_{20}|\Delta \lambda_t^\sS|+ e_{21}|\Delta \lambda_t^\sI|
\Bigr] \\
&\leq e_{12}|\Delta \psi_t^{\sS}|^2 + (e_{12}+2e_{13}+\cdots+e_{21})|\Delta \psi_t^{\sI}|^2 + e_{14}|\Delta \varphi_t^{\sS}|^2 + e_{15}|\Delta \varphi_t^{\sI}|^2 + \tilde{e}_{16}\|\Delta p_t\|^2 + \tilde{e}_{17}\|\Delta u_t\|^2 + \tilde{e}_{18}\|\Delta \theta_t\|^2,
\end{aligned}
$$
where $e_{12}:=\beta+\beta^2 \bar{U}$, $e_{13}:=\gamma$, $e_{14}:=\frac{1}{2}\beta^2$, $e_{15}:=\frac{1}{2}\beta^2$, $e_{16}:=\beta^2 \bar{\varphi}$, $e_{17}:=2\beta^2\bar{\varphi}+2\beta^2\bar{\psi}\bar{U}+\beta \bar{\psi}$, $e_{18}:=\frac{1}{2}\beta^2 \bar{\psi}$, $e_{19}:=\frac{1}{2}\beta^2 \bar{\psi}$, $e_{20}:=\beta \bar{\psi}$, $e_{21}:=2\beta^2\bar{\varphi}+2\beta^2\bar{\psi}\bar{U}+\beta \bar{\psi}$. We have for state $\sR$
$$
\begin{aligned}
\frac{d}{dt} |\Delta \psi_t^{\sR}|^2
= 2\Delta \psi_t^{\sR}\Delta \dot \psi_t^{\sR} & \leq 2|\Delta \psi_t^{\sR}|
\big[
V|\Delta \psi_t^{\sS}| + \gamma|\Delta \psi_t^{\sI}| + \eta|\Delta \psi_t^{\sR}| + \bar{\psi}|\Delta v_t|
\big] \\
&\leq V|\Delta \psi_t^{\sS}|^2 + \gamma|\Delta \psi_t^{\sI}|^2 + (V+\gamma+2\eta+\bar{\psi})|\Delta \psi_t^{\sR}|^2 + \bar{\psi}|\Delta v_t|^2.
\end{aligned}
$$
Combining the three estimates yields
$$
\frac{d}{dt}\|\Delta \psi_t\|^2 \le E_1\|\Delta \psi_t\|^2 + E_2\|\Delta \varphi_t\|^2 + E_3\|\Delta p_t\|^2 + E_4\|\Delta u_t\|^2 + E_5\|\Delta \theta_t\|^2,
$$
where $E_1:=\max\{
  2e_1 + e_2 + \cdots + e_{11},
  e_{12} + 2e_{13} + \cdots + e_{21},
  V + \gamma + 2\eta + \bar{\psi}\}$, $E_2:=\max\{e_3+e_{14}, e_4+e_{15}\}$, $E_3:=\max\{\tilde{e}_5,\tilde{e}_{16}\}$, $E_4:=\max\{\tilde{e}_6,\tilde{e}_{17}\}$, $E_5:=\max\{\tilde{e}_7, \tilde{e}_{18},\bar{\psi}\}$. Since $\Delta \psi_0 = 0$, Gr\"onwall's inequality gives
$$\|\Delta \psi_t\|_2^2 \leq \int_0^t e^{E_1(t-s)}
\big[
E_2\|\Delta \varphi_s\|^2 + E_3\|\Delta p_s\|^2 + E_4\|\Delta u_s\|^2 + E_5\|\Delta \theta_s\|^2
\big] ds.$$
Therefore, $\Delta \Psi \le e^{E_1T}T\big[E_2\Delta \Phi + E_3\Delta P + E_4\Delta U + E_5\sup_{t\in[0,T]}\|\Delta \theta_t\|^2\big]$. Plugging the above estimate into the bound for $\Delta \Phi$ yields
\begin{equation*}
    \begin{aligned}
        \Delta \Phi & \leq e^{(D_1+E_1)T}D_2E_2T^2\Delta \Phi + [e^{(D_1+E_1)T}D_2E_3T^2 + e^{D_1T}D_3T]\Delta P \\ & \quad \quad + [e^{(D_1+E_1)T}D_2E_4T^2 + e^{D_1T}D_4T]\Delta U + e^{D_1T}D_5T\Delta \chi + [e^{(D_1+E_1)T}D_2E_5T^2 + e^{D_1T}D_6T]\sup_{t\in[0,T]}\|\Delta \theta_t\|^2
    \end{aligned}
\end{equation*}
For short enough time horizon such that $e^{(D_1+E_1)T}D_2E_2T^2 < 1$,
we obtain $\|\Delta \varphi\|_T \leq C_\varphi \|\Delta \theta\|_T$
for some constant $C_{\varphi} > 0$. Substituting this back into the estimate for $\Delta \Psi$ gives
$\|\Delta \psi\|_T \leq C_\psi \|\Delta \theta
\|_T$ for some constant $C_\psi > 0$. Therefore, $h_3$ is Lipschitz continuous.

\textbf{Step 4. Estimation on $h_4$.}
Assume the short-time condition in Lemma~\ref{lemma:h_convexity} holds and adopt its notation. For $i=1,2$, we denote
\small
$$
M_t^i:=
\begin{pmatrix}
K^\sS & q_t^i\\
q_t^i & n_t^i
\end{pmatrix},
\qquad
r_t^i:=
\begin{pmatrix}
K^\sS\bar\lambda^\sS-m^\sS p_t^{i,\sS}\chi_t^i\\
K^\sI\bar\lambda^\sI-m^\sI p_t^{i,\sI}\chi_t^i
\end{pmatrix}.
$$\normalsize
Then $\tilde{\lambda}_t^{i,\sS}$ and $\tilde{\lambda}_t^{i,\sI}$ in \eqref{eq:smfe_control} can be written as $(M_t^i)^{-1}r_t^i$. By Lemma~\ref{lemma:h_convexity}, under small enough time horizon $\det(M_t^i)\geq \frac{1}{2}K^\sS K^\sI$ for all $t\in[0,T]$ and $i=1,2$. Since the entries of $M_t^i$ are uniformly bounded, there exists a constant $\bar C_M>0$ such that $\|(M_t^i)^{-1}\|\leq \bar{C}_M, \forall t\in[0,T], i=1,2$. Using the uniform bounds on $(p,u,\phi,\psi,\chi)$, there exist constants
$\bar C_q,\bar C_n,\bar C_r>0$ such that for all $t\in[0,T]$,
\begin{equation*}
    \begin{aligned}
        |q_t^1-q_t^2|
&\leq \bar C_q\big(\|\Delta p_t\|_2+\|\Delta u_t\|_2+\|\Delta\phi_t\|_2+\|\Delta\psi_t\|_2\big), \\ |n_t^1-n_t^2|
&\leq \bar C_n\big(\|\Delta p_t\|_2+\|\Delta u_t\|_2+\|\Delta\phi_t\|_2+\|\Delta\psi_t\|_2\big), \\
\|r_t^1-r_t^2\|_2
&\leq \bar C_r\big(\|\Delta p_t\|_2+|\Delta\chi_t|\big).
    \end{aligned}
\end{equation*}

Using the identity $A^{-1}-B^{-1}=A^{-1}(B-A)B^{-1}$, we obtain
$$
\begin{aligned}
\|\Delta\tilde{\Lambda}_t\|_2
&:=\|(M_t^1)^{-1}r_t^1-(M_t^2)^{-1}r_t^2\|_2\\
&\leq \|(M_t^1)^{-1}\|\,\|r_t^1-r_t^2\|_2
+\|(M_t^1)^{-1}-(M_t^2)^{-1}\|\,\|r_t^2\|_2\\
&\leq C_{\Lambda}\big(\|\Delta p_t\|_2+\|\Delta u_t\|_2+\|\Delta\phi_t\|_2+\|\Delta\psi_t\|_2+|\Delta\chi_t|\big)
\end{aligned}
$$
for some constant $C_{\Lambda}>0$. For $\lambda^\sR$, we directly have
$$|\Delta\tilde\lambda_t^R|
\leq \frac{m^\sR}{K^\sR}(\bar{\chi}|\Delta p_t^\sR|+|\Delta\chi_t|)
\leq C_R(\|\Delta p_t\|_2+|\Delta\chi_t|).$$
For $v$, by the uniform boundedness of $(p,u,\phi,\psi,\chi)$,
$$|\Delta\tilde v_t|
\le C_v(\|\Delta p_t\|_2+\|\Delta u_t\|_2+\|\Delta\phi_t\|_2+\|\Delta\psi_t\|_2+|\Delta\chi_t|)$$
for some constant $C_v>0$. Combining the above estimates, there exists a constant $C_4>0$ such that
$$\|\Delta\tilde\theta_t\|_2
\le C_4\big(\|\Delta p_t\|_2+\|\Delta u_t\|_2+\|\Delta\phi_t\|_2+\|\Delta\psi_t\|_2+|\Delta\chi_t|\big),
\quad \forall t\in[0,T].$$
Taking supremum over $t\in[0,T]$ and using the estimates from Steps 1--3, we obtain
$$\|f(\theta_1)-f(\theta_2)\|_T
=\|\Delta\tilde\theta\|_T
\le C_4\big(C_p+C_u+C_\phi+C_\psi+C_\chi\big)\|\theta_1-\theta_2\|_T.$$
We may choose $T>0$ sufficiently small such that $C:=C_4\big(C_p+C_u+C_\phi+C_\psi+C_\chi\big)<1$, then $f$ is a contraction on $\mathcal{M}$. Banach fixed point theorem implies that there exists a unique fixed point $\theta^*\in \mathcal{M}$. The corresponding trajectories $(u,p,B,\phi,\psi,\chi)$ obtained through Steps 1-3 then provide the unique bounded continuous solution to the FBODE system \eqref{eq:smfe_fbode}. This completes the proof.

\end{document}